\documentclass[12pt,reqno]{amsart}

\usepackage[utf8]{inputenc}

\usepackage[a4paper,top=3cm,bottom=2cm,left=3cm,right=3cm,marginparwidth=1.75cm]{geometry}

\usepackage{amsmath}
\usepackage{amssymb}
\usepackage{amsthm}
\usepackage{enumitem}

\theoremstyle{definition}

\newtheorem{exa}{Example}[section]

\theoremstyle{plain}
\newtheorem{thm}{Theorem}[section]

\newtheorem{prop}{Proposition}[section]
\newtheorem{cor}{Corollary}[section]

\theoremstyle{remark}
\newtheorem*{rmk}{Remark}

\newtheorem*{notation}{Notation}

\DeclareMathOperator{\dist}{dist}

\newcommand{\R}{\mathbb{R}}

\newcommand{\T}{\mathbb{T}}
\newcommand{\A}{\mathbb{A}}
\newcommand{\II}{\mathcal{I}}

\usepackage[myheadings]{fullpage}
\usepackage{fancyhdr}
\usepackage{lastpage}
\usepackage{graphicx, wrapfig, setspace}
\usepackage[T1]{fontenc}
\usepackage[english]{babel}

\title{An Incidence Result for Well-Spaced Atoms in all dimensions}
\author{Peter Bradshaw}

\begin{document}

\maketitle

\begin{abstract}
    We prove an incidence result counting the $k$-rich $\delta$-tubes induced by a well-spaced set of $\delta$-atoms. Our result coincides with the bound that would be heuristically predicted by the Szemer\'edi--Trotter Theorem and holds in all dimensions $d \geq 2$.
\end{abstract}

\section{Introduction}
Incidence geometry is concerned with counting incidences between various geometric objects, traditionally points and lines. Given a set of lines $L$ in $\R^2$, the set of points through which at least $k$ lines pass is denoted by $N_k(L)$. Then for $k\geq 2$, the classical Szemer\'edi--Trotter Theorem \cite{SzT} bounds sharply $|N_k(L)|$: 
\[ |N_k(L)| \ll |L|^2 k^{-3} + |L| k^{-1}. \]
By duality, the above bound also holds if the roles of points and lines are interchanged.
Given a set $P$ of points and a set $L$ of lines in $\R^2$, the set of incidences is defined to be $\II(P,L) := \{(p,l) \in P\times L: p\in l\}$. An equivalent formulation of the Szemer\'edi--Trotter Theorem states that
\begin{equation}\label{sztr}
  \II(P,L) \ll (|P||L|)^{2/3} + |P| + |L|.  
\end{equation}

In 2019, Guth, Solomon and Wang proved that given a small $\delta$ and for a suitably well-spaced set $L$ of tubes of thickness $\delta$ in $[0,1]^2$, an analogue of the Szemer\'edi--Trotter Theorem holds \cite{tubes}. Furthermore, they proved a similar result in $[0,1]^3$ which is an analogue of the seminal Guth--Katz bound \cite{PHSThm}. Both bounds are essentially sharp.  

As in \cite{tubes}, our objects of interest will be small $\delta$-atoms and thin $\delta$-tubes. A $\delta$-atom is a closed ball in $[0,1]^d$ of diameter $\delta$ and a $\delta$-tube is the set of all points in $[0,1]^d$ which are within a distance of $\delta/2$ from some fixed line.
Unlike the discrete setting of points and lines, we need to carefully define what it means for two atoms or two tubes to be distinct. 
Two $\delta$-atoms are distinct if they do not intersect each other. Two $\delta$-tubes are considered distinct if either:
\begin{itemize}
    \item They do not intersect each other, or;
    \item The angle between them is $>\delta$.
\end{itemize}
We say that an atom and a tube are \emph{incident} with each other if they have a non-empty intersection.

These assumptions are sensible because they avoid degenerate situations such as many atoms being incident with the exact same set of tubes, or many tubes being incident with the exact same set of atoms.

However, imposing these assumptions is still insufficient for proving a result similar to the Szemer\'edi--Trotter Theorem. The most fundamental property of points and lines which is utilised in proving discrete incidence results is that two points lie on at most one line and two lines intersect in at most one point.
However, this does not hold for atoms and tubes, and constitutes one of the most important differences between the two cases. In fact if two $\delta$-atoms in $[0,1]^d$ are separated by a distance of $x$ where $\delta \ll x < 1$, then there exist $\sim x^{1-d}$ distinct tubes which are incident to both of them.

Let $T_k(A)$ be the set of $k$-rich tubes induced by the $\delta$-atoms $A$. 
The following example illustrates that without further assumption on the distribution of atoms, useful bounds cannot be obtained.
Suppose $A$ is a set of $k^2$ $\delta$-atoms arranged next to each other in a $k\times k$ square grid in one corner of $[0,1]^2$.
It is relatively clear that $|T_k(A)| = k\delta^{-1} = \delta^{-1}\cdot \frac{|A|^2}{k^3}$. Since $\delta^{-1}$ can be arbitrarily large, the result is not useful.

We will show that if $A$ is well-distributed in some sense, then a bound for $|T_k(A)|$ can be obtained which essentially depends only on $|A|$ and $k$.
Our methods which are inspired by \cite{tubes} allow us to bound $|T_k(A)|$, where $A$ is a well-distributed (or well-spaced) set of atoms in $[0,1]^d$. In this context, well-distributed means that the atoms almost lie in a grid.

\begin{thm}\label{thm}
Let $A$ be a family of $W^{d}$ $\delta$-atoms in $[0,1]^d$, with $1<W<\delta^{-1}$, such that every $W^{-1}\times \dots\times W^{-1}$ cube contains $O(1)$ $\delta$-atoms from $A$. 
Let $k \geq 2$.
Then for every $\epsilon > 0$ there exist $C_1(\epsilon,d)$ and $C_2(\epsilon,d)$ such that if
\begin{equation}\label{condition}
    k\geq C_1(\epsilon, d)\delta^{-\epsilon}\cdot \delta^{d-1}|A|,
\end{equation}
then
\begin{equation}\label{result}
    |T_k(A)| \leq C_2(\epsilon,d) \delta^{-\epsilon} \cdot \frac{|A|^2}{k^3}.
\end{equation}
\end{thm}

Given a general set $T$ of tubes, $\mathcal{I}(A,T)$ will be the number of incidences between atoms from $A$ and tubes from $T$.
 Concretely,
\[\mathcal{I}(A,T) = |\{(a,t)\in A \times T: a\cap t \neq \emptyset\}|.\]
We can obtain an equivalent formulation of Theorem \ref{thm} in terms of incidences.
Firstly order the tubes $T$ by decreasing richness so that each tube $t_i\in T$ is $k_i$-rich and $\{k_i\}_{i\geq 1}$ is decreasing. Then rearranging \eqref{result} and summing over all $k_i$ gives the following:

\begin{cor}\label{cor}
Let $A$ be a family of $W^{d}$ $\delta$-atoms in $[0,1]^d$, with $1<W<\delta^{-1}$, such that every $W^{-1}\times \dots\times W^{-1}$ cube contains $O(1)$ $\delta$-atoms from $A$. 
Let $T$ be an arbitrary set of $\delta$-tubes.
Then for every $\epsilon > 0$, there exists $C_3(\epsilon,d)$, such that
\begin{equation}\label{inc_result}
    \II(A,T) \leq C_3(\epsilon,d) \delta^{-\epsilon}(|A|^{2/3}|T|^{2/3} + k_0(A,\delta)|T|),
\end{equation}
where $k_0(A,\delta) := \max\{1,\delta^{d-1}|A|\}$.
\end{cor}
The term $k_0(A,\delta)|T|$ in \eqref{inc_result} plays the same role as the $|L|$ term in the Szemer\'edi--Trotter bound \eqref{sztr}, namely counting the incidences from lines incident to only one point.

\begin{rmk}
The condition \eqref{condition} on $k$ is necessary and has a specific meaning. If the atoms in $A$ were randomly placed, then a simple calculation verifies that the expected richness of any $\delta$-tube is $\delta^{d-1}|A|$. Suppose $k$ is less than this threshold. Probabilistic arguments show that there exist configurations of atoms $A$ such that a positive proportion of all tubes are $k$-rich, so non-trivial results cannot be found. 
\end{rmk}

If $d = 2$, Theorem \ref{thm} follows by duality from the result in \cite{tubes} and is essentially optimal. When $d\geq 3$, the conjectured bound would have $k^3$ replaced with $k^{d+1}$ in the denominator of \eqref{result}, but any improvement towards this appears not to be amenable to this method. In the discrete setting, our result can be obtained in any dimension $d\geq 3$ by projecting generically onto a plane and applying the Szemer\'edi--Trotter Theorem. However, this is not possible in the thickened setting because projecting atoms into a plane will not preserve the well-distributed property of the set of atoms.

\begin{rmk}
It should be noted that it doesn't matter if we consider our atoms to be $d$-dimensional spheres or cubes or some (slightly) more exotic shape. There exist constants $C,c$ such that a $\delta$-cube is contained in a $C\delta$-sphere, and also contains a $c\delta$-sphere. Since we are primarily interested in growth rate, the introduced constants are of no consequence. 
During our proof, we will partition the space $[0,1]^d$ into ``cells'' which is most natural if we view our cells as smaller cubes.
One important upshot is that all equalities in this paper are implicitly up to an absolute constant.  None of these constants are problematically large or small.
\end{rmk}

\begin{notation}
We use Vinogradov's symbol extensively. We write $f(n) \ll g(n)$ to mean that there exists a constant $C$
such that $f(n) \leq Cg(n)$ for large $n$. 
Similarly is defined $\ll_{\alpha, \beta}$, but in this case the constant $C := C(\alpha,\beta)$ may depend on $\alpha$ and $\beta$.
\end{notation}

\begin{notation}
Throughout this paper, we say a tube is $k$-\emph{rich} if it passes through $k$ atoms.
We let $T_k(A)$ be the set of all $k$-rich tubes.
We will often suppress the notation to $T_k$ when context makes it clear what the set $A$ is. 
\end{notation}

\noindent{\bf Acknowledgement:} The author would like to thank Misha Rudnev for his invaluable help and guidance throughout the preparation of this paper.

\section{A general incidence result}

To assist in our proofs, we will mostly be working with incidence counts rather than directly with $k$-rich tubes. Furthermore, we allow for \emph{weighted} sets of atoms and tubes.
Let $A$ be a set of atoms where each $a\in A$ has a positive integer weight $w(a)$ associated with it.
This essentially means that when counting incidences, the atom $a$ appears $w(a)$ times.
Similarly define weight for sets of tubes.

For a set of weighted atoms $A$ with weight function $w$, and a set of weighted tubes $T$ with weight function $\omega$, we define a more general incidence counting function
\[ \II(A,T) := \sum_{a\in A}\sum_{t\in T} w(a) \omega(t) \mathbb{I}_{\{a \cap t \neq \emptyset\}}. \]

The following incidence result has the same role as Proposition 2.1 in \cite{tubes} and our proof is a modification of theirs. This result is necessary in the proof of Theorem \ref{thm} because it can be used for any sets of atoms and tubes, even if they are not well-spaced.

\begin{prop}\label{prop}
Let $k \geq 1$ and $A$ be a set of distinct weighted $\delta$-atoms in $[0,1]^d$ with weight function $w$. Let $T$ be a set of distinct (not weighted) $\delta$-tubes. Let $S \in (1,\delta^{-1})$. Then 
\begin{equation}\label{propeq}
    \II(A,T) \ll \left(S\delta^{-(d-1)}|T|\sum_{a\in A}w(a)^2\right)^{1/2} + S^{1-d}\II(A^S,T^S),
\end{equation}
where $A^S$ and $T^S$ are respectively the weighted sets of atoms and tubes formed by thickening $A$ and $T$ by a factor of $S$.
\end{prop}

\begin{proof}
We scale the problem by $\delta^{-1}$, so that the atoms are now $1$-atoms in $[0,\delta^{-1}]^d$. This will be more convenient to work with.

For any $a \in A$ and any $t\in T$, let $\chi_a(x)$ and $\chi_t(x)$ respectively be the indicator functions for the atom $a$ and the tube $t$. Now let $f(x) := \sum_{a\in A} w(a)\chi_a(x)$ and $g(x) := \sum_{t\in T} \chi_t(x)$. 

Let $h$ be $S^{-d}$ times a smooth $C^\infty$ function which approximates the indicator function for an $S$-atom with centre at the origin. The precise function $h$ is chosen so that its Fourier transform is supported on a box and decays rapidly outside this box. The importance of multiplying by $S^{-d}$ is that it ensures that the value of an integral of a function is not changed by convolving with $h$. In fact, convolving $f$ with $h$ thickens atoms by a factor of $S$ and flattens them by a factor of $S^{-d}$, and convolving $g$ with $h$ thickens the tubes by $S$ (and makes them slightly longer, but this is irrelevant) and flattens them by a factor of $S^{-(d-1)}$.

Counting the incidences, we get
\[\II(A,T) \sim \int_{[0,1]^d} f(x)g(x)dx. \]

We thicken all of the $1$-tubes to get weighted $S$-tubes. (They are weighted because some tubes may coincide after thickening.) This thickening procedure corresponds to convolving $g$ with $h$. 
Also note that whether we thicken atoms, tubes, or both atoms and tubes, the effect on the incidence count is the same. That is
\[\int (f \ast h)(x)g(x)dx = \int f(x)(g\ast h)(x)dx = \int (f\ast h)(x)(g\ast h)(x)dx.\]

Firstly suppose that $\int fg \ll \int(f\ast h)(g\ast h)$. This is clearly equivalent to $\II(A,T) \ll S^{1-d} \II(A^S,T^S)$, proving one half of the bound.

Otherwise we have
$\int fg \gg \int f(g\ast h)$, and hence
\[\II(A,T) = \int fg \ll \int fg-f(g\ast h).\]

We know by Plancharel's theorem that $\int f(x)g(x)dx = \int \hat{f}(\xi)\hat{g}(\xi)d\xi$. 
Since convolution becomes multiplication when we pass to the Fourier transform, it follows that
\[\II(A,T) = \int f (g-(g\ast h)) = \int \hat{f} \overline{\hat{g}} (1-\eta),\]
where $\eta$ is $1$ in an $S^{-1}$-ball around the origin, and decays quickly outside.
Using the Cauchy--Schwarz inequality, we get
\begin{equation} \label{cs_step}
\int \hat{f}(\xi)\hat{g}(\xi) (1-\eta(\xi))d\xi \leq \left(\int |\hat{f}(\xi)|^2 d\xi\right)^{1/2} \left(\int |\hat{g}(\xi)|^2 (1-\eta(\xi))^2 d\xi\right)^{1/2}.    
\end{equation}

Using Parseval's identity, the first term on the right-hand side can be evaluated as
\[\left(\int |\hat{f}(\xi)|^2 d\xi\right)^{1/2} = \left(\int |f(x)|^2 dx\right)^{1/2} = \left(\sum_{a\in A} w(a)^2\right)^{1/2}.\]

We now estimate the second term in \eqref{cs_step}. We partition the surface of $[0,1]^d$ into small $\delta$-caps as a way of sorting tubes in $T$ by direction. Let $T_\theta$ be the set of all $k$-rich tubes in the direction of $\delta$-cap $\theta$, and let $g_\theta = \sum_{t \in T_\theta} \chi_t$. 
If $t$ is a $1$-tube passing through the origin in direction $\theta$, 
then its Fourier transform is roughly $\delta^{-1}$ times the indicator function of a $1\times \dots \times 1 \times \delta$ slab orthogonal to $t$ with centre at the origin.

We assume that $\xi \geq S^{-1}$, as the integral is otherwise zero. For a fixed $\xi$, there is at most one tube in each direction such that $\hat{\chi_t}(\xi)$ is nonzero. Furthermore, simple geometric arguments show that
the contribution of $\hat{g_\theta}(\xi)$ to $\hat{g}(\xi)$ is only nontrivial for at most $\ll S\delta^{-(d-2)}$ different $\theta$ values. We then apply the Cauchy--Schwarz inequality to get
\[ (1-\eta(\xi))^2|\hat{g}(\xi)|^2 = (1-\eta(\xi))^2|\sum_\theta \hat{g_\theta}(\xi)|^2 \ll S\delta^{-(d-1)} \sum_{\theta} |\hat{g_\theta}(\xi)|^2.\]
Again using Parseval's identity, it follows that
\[\int |\hat{g}(\xi)|^2 (1-\eta(\xi))^2 d\xi \ll S\delta^{-(d-2)} \sum_\theta \int |\hat{g_\theta}(\xi)|^2 = S\delta^{-(d-2)} \sum_\theta \int |g_\theta(x)|^2 = S\delta^{-(d-1)}|T|.\]
Substituting into \eqref{cs_step} yields
\[\II(A,T) \ll \left(S\delta^{-(d-1)}|T|\sum_{a\in A}w(a)^2\right)^{1/2},\]
the other half of the bound.
\end{proof}

The dominant term in \eqref{propeq} is determined based on whether the incidence count increases disproportionately after thickening by $S$. The following two examples give configurations of atoms and tubes which attain both bounds in Proposition \ref{prop}, demonstrating that it is sharp up to a factor of $S$. For the purpose of these examples, $A$ will not be a weighted set of atoms. 

\begin{exa}
If $A \subset [0,1]^d$ consists of all the $\delta$-atoms in a $d$-dimensional box with side length $k\delta$, then $|A| = k^d$. If $T$ is the set of induced $k$-rich $\delta$-tubes, it can be shown that $|T| = \delta^{-(d-1)}k^{d-1}$. 
Further calculations show that 
\[ (S\delta^{-(d-1)} |A||T|)^{1/2} = S^{1/2} \delta^{-(d-1)} k^{d-\frac1{2}} \qquad \text{and} \qquad S^{1-d}\II(A^S, T^S) = \delta^{-(d-1)}k^d.\]
Also since all tubes in $T$ are $k$-rich, we have
\[ \II(A,T) = \delta^{-(d-1)}k^d,\]
so the second term in \eqref{propeq} is the attained bound (for a small enough choice of $S$).
\end{exa}

\begin{exa}
If $A \subset [0,1]^d$ consists of a $(d-1)$-dimensional slice of the above configuration of $\delta$-atoms, then we have $|A| = k^{d-1}$. Again let $T$ be the set of induced $k$-rich $\delta$-tubes, so $|T| = \delta^{-(d-1)}k^{d-3}$. In this case
\[ (S\delta^{-(d-1)} |A||T|)^{1/2} = S^{1/2} \delta^{-(d-1)} k^{d-2} \qquad \text{and} \qquad S^{1-d}\II(A^S, T^S) = S^{-1}\delta^{-(d-1)}k^{d-2}.\]
Again, since all tubes in $T$ are $k$-rich, we have
\[ \II(A,T) = \delta^{-(d-1)}k^{d-2},\]
so the first term in \eqref{propeq} is the attained bound up to an $S^{1/2}$ factor.
\end{exa}

\section{The Main Result}

The proof of Theorem \ref{thm} combines induction with a cell partitioning argument.
Often in proofs of incidence results it is useful to partition the space into smaller cells and estimate the contribution of incidences in each cell. An illustrative example is a very short, elementary proof of the Szemer\'edi--Trotter Theorem for cartesian products using a ``lucky pairs'' argument. The prototype for this method can be found in \cite{szem-tr_cart}. 

The lucky pairs argument adapts readily to higher dimensions, and since our set of atoms $A$ in $[0,1]^d$ is nearly a $d$-fold cartesian product, we conjecture that the corresponding bound should hold, namely that $|T_k(A)| \ll |A|^2 k^{-(d+1)}$.

We again mention that the higher-dimensional version of Theorem \ref{thm} does not follow from projecting into the plane and applying the $d=2$ result, since the well-distributed assumption that $A$ is nearly grid-like is clearly violated after projection.

Our strategy is the following: We partition $[0,1]^d$ into cells of side length $D^{-1}$ for some parameter $D$ to be chosen. 
Proposition \ref{prop} with some thickening parameter $S$ allows us to relate the number of $k$-rich tubes to an incidence count, specifically the $L_2$-norm of the weights of shortened tubes in all cells. This is bounded by applying the induction hypothesis in each cell. The method is inspired by \cite{tubes}, but is different in several key ways. The exposition is also new.

For the proof to work, we need $S$ to be much smaller than $D$, and $D$ to be much smaller than $\delta^{-\epsilon}$. We also want $S$ and $D$ to be much bigger than constants. This is the motivation for the uniform choices of these parameters given in the proof.

\begin{proof}[Proof of Theorem \ref{thm}]
We treat $\epsilon$ and $d$ as constants, so in what follows $\ll$ is written to mean $\ll_{\epsilon,d}$. 
We fix $W$ and proceed by induction on $\delta$. Namely we have to prove the statement for all $\delta \in (0,W^{-1})$. The first base case will be when $\delta$ is very close to $W^{-1}$, namely when $\delta^{-(1-c\epsilon)} \leq W$ for some small fixed $c$ (we choose $c<1/(d-1)$ which assists in the following calculation).
Assuming $\delta^{-(1-c\epsilon)} \leq W$, \eqref{condition} gives
\[ k \geq C_1(\epsilon,d) W^{d - \frac{d-1-\epsilon}{1-c\epsilon}} > C_1(\epsilon,d) W.\]
Since the distribution of atoms permits $|T_k(A)|$ to be non-zero only if $k\ll W$, we can choose $C_1({\epsilon},d)$ large enough so that $|T_k(A)| = 0$, and \eqref{result} holds trivially.

The other base case is when $W$ is very small, say smaller than some constant $c$. In this case, $|T_k(A)| \leq c^{2d}$ trivially, so \eqref{result} holds for a suitable choice of $C_2(\epsilon,d)$.

We move on to the induction step. Assume the result holds for all $\delta' > K\delta$ and $W'^{-1} > KW^{-1}$, where $K$ is sufficiently small ($K = 2$ will work). 
Assume \eqref{condition} holds. 

Firstly, we split up $[0,1]^d$ into $D^d$ sub-cubes or \emph{cells}, where $D  = \delta^{-c^2 \epsilon^2}$. Let a \emph{tubechen} be the intersection of a $k$-rich $\delta$-tube with one of these cells. (A tubechen looks like a section of a $\delta$-tube of length $D^{-1}$.)
To each tubechen $t$ we associate a weight $w(t)$ which is the number of atoms from $A$ in $t$, and a multiplicity $m(t)$ which is the number of $k$-rich tubes containing the tubechen $t$. (A tubechen is ``contained'' in a tube if all the atoms on the tubechen also intersect the tube. A tubechen may lie on up to $D^{d-1}$ $k$-rich tubes.)
With this notation, it is evident that
\[ k|T_k| \leq \II(A,T_k) = \sum_{\text{tubechens }t} w(t)m(t). \]
It is also clear that 
\[\sum_{\text{tubechens }t} m(t) = D|T_k|,\]
so by the pigeonhole principle, a positive proportion of the incidences come from tubechens $t$ with $w(t) \geq k/D$. We will henceforth assume that the weights of all tubechens are at least $k/D$.

At this point we again emphasise that our choices for the shapes and sizes of atoms and tubes makes our definition of \emph{incidence} quite loose. For this reason, equations such as those in the above pigeonholing argument contain suppressed constants. We will henceforth use mostly asymptotic notation.

Now we treat separately two cases: $k \ll D$ and $k \gg D$.
The reason is that we will later apply the induction hypothesis to estimate the number of $k/D$-rich tubechens, and the induction hypothesis only holds if $k/D$ is greater than some constant. 

{\bf Case $1$: $k\ll D$.}
Since $D = \delta^{-c^2\epsilon^2}$, it follows that $\delta^{-\epsilon/2}k^{-3} \gg 1$.
Now let's count the $2$-rich tubes. For each pair of atoms $a,a'\in A$, there are $\dist(a,a')^{-(d-1)}$ tubes passing through both. Then the configuration imposed by the spacing assumptions allows us to make the following bounds:
\[|T_k(A)| \leq |T_2(A)| = \sum_{a,a' \in A} \dist(a,a')^{-(d-1)} \ll |A|^2 \log |A| \ll \delta^{-\epsilon} |A|^2 k^{-3}.\]

{\bf Case $2$: $k \gg D$.} 
We want to apply Proposition \ref{prop} using a thickening factor $S = \delta^{-c^3\epsilon^3}$,
but to do so globally is wasteful of the strong spacing assumptions on $A$, so the bound will be prohibitively weak. 
We first partition all the tubes in $[0,1]^d$ into $(D\delta)^{-2(d-1)}$ $D\delta$-tubes. 
The rationale for this partitioning is that it will make tubechens behave like weighted atoms.

If we fatten a $D\delta$-tube $\tau$ by a factor of $D^{-1}\delta^{-1}$, it becomes $[0,1]^d$. Each tubechens in $\tau$ which runs parallel to $\tau$ becomes  a $D^{-1}$-atom, and each $k$-rich $\delta$-tube in $\tau$ becomes a $D^{-1}$-tube.
Furthermore, each new atom has a weight, which is the same as the weight of the corresponding tubechen.
Call this set of new $D^{-1}$-atoms $\A_\tau$ and the set of new $D^{-1}$-tubes $\T_\tau$.

Applying Proposition \ref{prop} in each $D\delta$-tube $\tau$, and applying Cauchy--Schwarz, we get
\begin{align}
    k|T_k| = \II(A,T_k) &= \sum_{\tau} \II(\A_\tau, \T_\tau) \nonumber \\
    &\ll \sum_{\tau}\left(SD^{d-1}|\T_\tau|\sum_{a\in \A_\tau} w(a)^2\right)^{1/2} + S^{1-d} \sum_\tau \II(\A_\tau^S, \T_\tau^S) \nonumber \\
    &\leq (SD^{d-1})^{1/2} \left(\sum_\tau |\T_\tau|\right)^{1/2} \left(\sum_\tau \sum_{a\in \A_\tau} w(a)^2\right)^{1/2} + S^{1-d} \II(A^S,T_k^S) \nonumber \\
    &= (SD^{d-1})^{1/2} |T_k|^{1/2} \left(\sum_\tau \sum_{a\in \A_\tau} w(a)^2\right)^{1/2} + S^{1-d} \II(A^S,T_k^S). \label{two_cases}
\end{align}
We will now have two cases based on which term in \eqref{two_cases} dominates. 

Firstly suppose the second term dominates. 
Since there is at most one atom in each $W^{-1}$-cell, all thickened $S\delta$-atoms in $A^S$ have weight one, and hence $|A^S|= |A|$.
Also, the weights of tubes in $T^S$ are trivially bounded above by $S^{2(d-1)}$, the maximum number of $\delta$-tubes contained in an $S\delta$-tube. If $\tilde{T}_k^S$ is the underlying set of \emph{unweighted} tubes, then 
\begin{equation}\label{dumb_bound}
  \II(A^S,T_k^S) \ll S^{2(d-1)}\II(A^S,\tilde{T}_k^S).  
\end{equation}
In order to have $k|T_k| \ll S^{1-d}\II(A^S,T_k^S)$, a positive proportion of these incidences must be supported on $S\delta$-tubes which are $S^{d-1}k$-rich in atoms from $A^S$.
Furthermore, \eqref{condition} implies that
\[S^{d-1}k \geq S^{d-1} C_1(\epsilon,d)\delta^{d-1-\epsilon}|A| \geq C_1(\epsilon,d)(S\delta)^{d-1-\epsilon}|A^S| \cdot S^\epsilon,\]
so we can apply the induction hypothesis for any richness $k' \geq S^{d-1}k$.
Standard dyadic summing of the induction hypothesis implies that
\begin{equation}\label{hiding_pigeonholing}
   \II(A^S,\tilde{T}_k^S) \ll (S\delta)^{-\epsilon} \frac{|A|^2}{(S^{d-1}k)^2}. 
\end{equation}
Then combining \eqref{two_cases}, \eqref{dumb_bound} and \eqref{hiding_pigeonholing} yields
\[ k|T_k| \ll S^{1-d-\epsilon} \delta^{-\epsilon} \frac{|A|^2}{k^2}, \]
and rearranging closes the induction.

Now assume the first term in \eqref{two_cases} dominates.
After rearranging, this implies that
\begin{equation}\label{a_bound}
    |T_k| \ll \frac{SD^{d-1}\left(\sum_\tau \sum_{a\in \A_\tau} w(a)^2\right)}{k^2}.
\end{equation}
Notice that the bracketed term in the numerator is a sum over all tubechens. A suitable bound on this quantity will complete the proof.

 Having already partitioned $[0,1]^d$ into $D^{d}$ cells, we now estimate the contribution from tubechens in each cell. For any of these cells $C$, let $A_C$ be the set of atoms from $A$ which lie in $C$ and let $T_C$ be the set of tubechens in $C$.

 If we enlarge each cell $C$ to $[0,1]^d$, then the $\delta$-atoms become $D\delta$-atoms which satisfy the conditions for applying the induction hypothesis. 
 Each tubechen $t$ is now a $D\delta$-tube, and the richness of this tube, denoted by $r(t)$, is the weight of the correponding tubechen. Recall that these weights all exceed $k/D$.
 For any $m > k/D$, using \eqref{condition} we get
 \begin{align*}
    m & > C_1(\epsilon,d) \delta^{d-1-\epsilon}|A| D^{-1} \\
    & > C_1(\epsilon,d)(\delta D)^{d-1-\epsilon} (D^{-d}|A|) \\
    & = C_1(\epsilon,d)(\delta D)^{d-1-\epsilon} |A_C|,
\end{align*}
so the induction hypothesis \eqref{result} can be used in any cell $C$ to bound $|T_m(A_C)|$.
We get
\begin{align*}
    \sum_{\tau} \sum_{a\in \A_\tau} w(a)^2 &= \sum_{C} \sum_{t\in T_C} r(t)^2 \\
    &=\sum_C \sum_{\substack{{m \text{ dyadic}} \\ {m= k/D}}}^{k} m^2 |T_m(A_C)| \\
    &\ll \sum_C \sum_{\substack{{m \text{ dyadic}} \\ {m= k/D}}}^{k} (D\delta)^{-\epsilon} (|A|D^{-d})^2 m^{-1} \\
    &\ll D^d (D\delta)^{-\epsilon} (|A|D^{-d})^2 \cdot(D/k). 
\end{align*}
Substituting this into \eqref{a_bound}, and recalling that $S = \delta^{-c^3\epsilon^3}$ and $D = \delta^{-c^2\epsilon^2}$, we get
\[ |T_k(A)| \ll \delta^{-\epsilon}|A|^2 k^{-3},\]
closing the induction and completing the proof.
\end{proof}

\begin{rmk}
There is a small omission in the above proof: In Case 2, it is essential that each cell contains approximately the same number of atoms from $A$ and that they are well-distributed. This follows immediately if $D < W$. But if $D \geq W$, then $\delta$ is so ridiculously small that the $\delta^{-\epsilon}$ factor in \eqref{result} is enormous, and trivial bounds give the desired result. Concretely, if $W \leq D = \delta^{-c^2\epsilon^2}$, then since no pair of atoms can lie on more than $W$ tubes,
\[|T_k(A)| \leq W|A|^2 \leq W^4 \cdot\frac{|A|^2}{k^3} \leq \delta^{-4c^2\epsilon^2} \cdot\frac{|A|^2}{k^3} \leq \delta^{-\epsilon} \cdot\frac{|A|^2}{k^3}.\]
\end{rmk}

\begin{rmk}
In \cite{tubes}, a bound is obtained for the number of $k$-rich atoms induced by a set of well-distributed tubes in $[0,1]^3$.
Our formulation can be adapted to get the same result. Firstly notice that a tube in $[0,1]^3$ can be parametrised by an atom in $[0,1]^4$. Furthermore, if the tubes are well-distributed in $[0,1]^3$, then the corresponding set of atoms will be well-distributed in $[0,1]^4$, and the $k$-rich atoms in $[0,1]^3$ become $k$-rich slabs ($\delta$-thickened planes) in $[0,1]^4$.
Similar parametrisations have been used to relate point-line incidences with point-plane incidences in the discrete setting. See for example \cite{points_planes}.

We partition $[0,1]^3$ into $D^3$ sub-cubes and let a \emph{planechen} be the intersection of a $k$-rich slab with one of these subcubes. 
By finding a suitable estimate for the incidence contributions of all planechens, we recover a bound on the number of $k$-rich slabs which is dual to the result in \cite{tubes}.
Slightly more care must be taken in the small $k$ case.
\end{rmk}


\begin{thebibliography}{9}
\bibitem{PHSThm}
L. Guth, N. H. Katz, \textit{On the Erdős distinct distance problem in the plane}, Annals of Mathematics 181 (2015), 155–190. 

\bibitem{tubes} 
L. Guth, N. Solomon, H.Wang,
\textit{Incidence estimates for well spaced tubes}, 
Geom. Funct. Anal. Vol.29 (2019), 1844-1863.

\bibitem{points_planes}
M. Rudnev, 
\textit{On the number of incidences between planes and points in three dimensions},
Combinatorica (2017), 1–36.

\bibitem{szem-tr_cart}
J. Solymosi, G. Tardos,
\textit{On the number of $k$-rich transformations}, Computational geometry (SCG'07), 227-231, New York, 2007

\bibitem{SzT}
E. Szemer\'edi, W. T. Trotter,\textit{ Jr.Extremal  problems  in  discrete  geometry}, Combinatorica 3 (1983), 381–392.



\end{thebibliography}
\end{document}